\documentclass{amsart}

\usepackage{mathrsfs}
\usepackage[dvips]{graphicx}
\usepackage{amsmath}
\usepackage[latin1]{inputenc}
\usepackage{amsfonts}
\usepackage{amssymb}
\usepackage{graphicx,epsfig}
\usepackage{amsthm}

\usepackage{amscd}
\usepackage{dsfont}
\usepackage[all,dvips]{xy}
\usepackage{fancyhdr}
\usepackage[T1]{fontenc}

\input cyracc.def

\newtheorem{theorem}{Theorem}[section]
\newtheorem{lemma}[theorem]{Lemma}
\newtheorem{proposition}[theorem]{Proposition}

\theoremstyle{definition}
\newtheorem{definition}[theorem]{Definition}
\newtheorem{example}[theorem]{Example}

\theoremstyle{remark}
\newtheorem{remark}[theorem]{Remark}

\allowdisplaybreaks

\numberwithin{equation}{section}

\begin{document}

\title{Quantum quasi-shuffle algebras II. Explicit formulas, dualization, and representations}

\author{Run-Qiang Jian}
\address{School of Computer Science, Dongguan University
of Technology, 1, Daxue Road, Songshan Lake, 523808, Dongguan, P.
R. China}
%\curraddr{}
\email{jianrq@dgut.edu.cn}
\thanks{}

%    General info

\date{}

%\dedicatory{}


\begin{abstract}
Using the concept of mixable shuffles, we formulate explicitly the
quantum quasi-shuffle product, as well as the subalgebra generated
by primitive elements of the quantum quasi-shuffle bialgebra. We
construct a braided coalgebra structure which is dual to the
quantum quasi-shuffle algebra. We provide representations of
quantum quasi-shuffle algebras on commutative braided Rota-Baxter
algebras. As an application, we establish formal power series
whose terms come from a special representation of the
quasi-shuffle algebra on polynomial algebra and whose evaluations at 1 are the multiple
$q$-zeta values.
\end{abstract}

\maketitle
\section{Introduction}
In \cite{Ree}, Ree introduced the shuffle algebra which has been
studied extensively during the past fifty years. The shuffle
product is carried out on the tensor space $T(V)$ of a vector
space $V$ by using the shuffle rule. Its natural generalization is
the quasi-shuffle product where $V$ is moreover an associative
algebra and the new product on $T(V)$ involves both of the shuffle
product and the multiplication of $V$. Quasi-shuffle algebras
first arose in the work of Newman and Radford \cite{NR} for the
study of cofree irreducible Hopf algebras built on associative
algebras, where they were constructed by the universal property of
cofree pointed irreducible coalgebras. Later, they were
rediscovered independently by other mathematicians with various
motivations. In 2000, motivated by his work on multiple zeta
values, Hoffman defined the quasi-shuffle algebra by an inductive
formula (\cite{Hof}). In the same year, Guo and Keigher introduced
the mixable shuffle algebra by using an explicit formula in their
study of Rota-Baxter algebras (\cite{G} and \cite{GK}). After
these seminal works, the research of quasi-shuffle algebras become
active. Besides their own interest, quasi-shuffle algebras have
many significant applications in other branches of mathematics,
such as multiple zeta values (\cite{IKZ}), Rota-Baxter algebras
(\cite{G} and \cite{EG}), and commutative tridendriform algebras
(\cite{Lod}). They also appear in the study of shuffle identities
between Feynman graphs (\cite{K}).

For both of physical and mathematical considerations, people want
to deform or quantize some important algebra structures. The most
famous example is absolutely the quantum group. To people's
surprise, there is an implicit but significant connection between
quantum groups and shuffle algebras. Rosso \cite{Ro} constructed
the quantization of shuffle algebras. This is a new kind of
quantized algebras and leads to an intrinsic understanding of the
quantum group. Since shuffle algebras are special quasi-shuffle
algebras, and the importance of the later one, people would expect
to find out what the quantization of quasi-shuffle algebras is and
whether it can bring us some useful information. Some
$q$-analogues of the quasi-shuffle product were discussed in
\cite{TU}, \cite{Hof} and \cite{B}. But these formulas are not
systematic. For instance, as Hoffman said in \cite{Hof}, his
$q$-deformation of quasi-shuffles is more or less experiential.
The general construction of quantized quasi-shuffles is due to
Rosso (\cite{Ro2}) in the spirit of his quantum shuffles. We
describe Rosso's idea as follows. Let $M$ be
 a Hopf bimodule over a Hopf algebra $H$. In addition, if $M$ is an algebra and the multiplication is
compatible with the module and comodule structure of $M$ in some
sense, then one can construct a new algebra structure on the
cotensor coalgebra $T^c_H(M)$ by using its universal property. Let
$M^R$ be the space of right coinvariants of $M$. Then the subspace
$T(M^R)$ equipped with this new multiplication is a generalization
of the classical quasi-shuffle algebra. In general, given a
braided algebra $(A,m,\sigma)$, one can construct an analogue of
the quasi-shuffle algebra on $T(A)$, where the action of the usual
flip is replaced by that of the braiding $\sigma$. The resulting algebra is
called a quantum quasi-shuffle algebra. In particular, the
$q$-analogues mentioned above are some sort of special cases of
Rosso's quantum quasi-shuffle. In \cite{JR}, the construction of
quantum quasi-shuffles appears, as a special braided cofree Hopf
algebra, in the framework of quantum multi-brace algebras.

Some interesting properties of quantum quasi-shuffle algebras have
been studied in \cite{JRZ}, including the commutativity, universal
property, and etc. In a recent paper \cite{J}, applications of
quantum quasi-shuffle algebras to Rota-Baxter algebras and
tridendriform algebras are found. This paper continues the trip.
We first establish some explicit results concerning this new
subject. We start by reformulating the product. Originally, the
quantum quasi-shuffle algebra is constructed by using the
universal property of connected coalgebras (\cite{JR}). Later, it
is defined through an inductive formula (\cite{JRZ}). But neither
of these two approaches can provide an explicit formula. To know
more about this new subject, a more clear form of the
multiplication formula is definitely helpful. Here, we use the
notion of mixable shuffles introduced in \cite{GK} to establish a
complete description of the quantum quasi-shuffle product. In the
case of quantum shuffles, the subalgebra generated by primitive
elements is especially important. Under some suitable assumptions,
it is isomorphic, as a Hopf algebra, to the positive part of
quantum groups (\cite{Ro}). It seems quite reasonable that the
corresponding subalgebra in a quantum quasi-shuffle algebra will
have some desirable properties. Recently, Fang and Rosso
(\cite{FR}) use it to realize the whole quantum group associated
to a symmetrizable Kac-Moody Lie algebra. In this paper, we use
mixable shuffles to describe such subalgebras. Sometimes, dual
constructions of algebraic objects bring people extra information
different from the original ones. On the other hand, the universal
property of connected coalgebras is not so familiar by non-Hopf
algebraists. So we use the universal property of tensor algebras
to construct a braided coalgebra structure on $T(C)$ for a braided
coalgebra $C$, and show that its dual is the quantum quasi-shuffle
algebra. This enables one to study the quantum quasi-shuffle
algebra through its dual. We would like to mention that Manchon
had studied such a structure when the braiding is the usual flip
map (\cite{Man}). Representation theory is absolutely an essential
tool for the investigation of algebras. Inspired by the work of
Guo and Keigher, we construct representations of quantum
quasi-shuffle algebras on some sort of braided Rota-Baxter
algebras which are introduced in \cite{J}. As an application, we
use some special kind of such representations to construct
formal power series whose terms are the images of an algebra map from
a quasi-shuffle algebra to a Rota-Baxter algebra. The evaluations
of these formal power series at 1 are the multiple $q$-zeta values.

This paper is organized as follows. In Section 2, the notion of braided algebra is recalled and several concrete
examples are provided. In Section 3, we recall
the construction of quantum quasi-shuffle algebras. After that, we
use the notion of mixable shuffles to establish an explicit
formula for the quantum quasi-shuffle product in Section 4 and describe
the subalgebra generated by primitive elements in Section 5. In
Section 6, we construct the dual coalgebra of a quantum
quasi-shuffle algebra. In Section 7, we construct algebra maps
from quantum quasi-shuffle algebras to commutative braided
Rota-Baxter algebras which provide representations of the former
ones. We use this construction to establish formal power series
whose evaluations at 1 are the multiple $q$-zeta values.

\noindent\textbf{Notation.} In this paper, we denote by
$\mathbb{K}$ a ground field of characteristic 0. All the objects
we discuss are defined over $\mathbb{K}$. For a vector space $V$,
we denote by $T(V)$ the tensor algebra of $V$, by $\otimes$ the
tensor product within $T(V)$, and by $\underline{\otimes}$ the one
between $T(V)$ and $T(V)$.

Let $\mathbb{N}$ be the set of positive integers. For any $n\in \mathbb{N}$, we denote by $\mathfrak{S}_{n}$ the symmetric group acting on the
set $\{1,2,\ldots,n\}$ and by $s_{i}$, $1\leq i\leq n-1$, the
standard generators of $\mathfrak{S}_{n}$ permuting $i$ and $i+1$. For any permutation $w\in \mathfrak{S}_{n}$, we usually write it by its two-line form, i.e., \[ w=\left(\begin{array}{cccc}
1&2&\cdots&n\\
w(1)&w(2)&\cdots&w(n)
\end{array}\right).
\]
For fixed $k,n\in \mathbb{N}$, we define the shift map
$\mathrm{shift}_k:\mathfrak{S}_{n}\rightarrow\mathfrak{S}_{n+k}$
by $\mathrm{shift}_k(s_i)=s_{i+k}$ for any $1\leq i\leq n-1$. For
the reason of intuition and the simplicity of notation, we denote
$1_{\mathfrak{S}_{k}}\times w=\mathrm{shift}_k(w)$ for any $w\in
\mathfrak{S}_{n}$. The notations $w\times 1_{\mathfrak{S}_{k}}$,
$1_{\mathfrak{S}_{k}}\times w\times 1_{\mathfrak{S}_{l}}$ and
others are understood similarly.

A braiding $\sigma$ on a vector space $V$ is an invertible linear
map in $\mathrm{End}(V\otimes V)$ satisfying the quantum
Yang-Baxter equation on $V^{\otimes 3}$: $$(\sigma\otimes
\mathrm{id}_{V})(\mathrm{id}_{V}\otimes \sigma)(\sigma\otimes
\mathrm{id}_{V})=(\mathrm{id}_{V}\otimes \sigma)(\sigma\otimes
\mathrm{id}_{V})(\mathrm{id}_{V}\otimes \sigma).$$ A braided
vector space $(V,\sigma)$ is a vector space $V$ equipped with a
braiding $\sigma$. For any $n\in \mathbb{N}$ and $1\leq i\leq
n-1$, we denote by $\sigma_i$ the operator $\mathrm{id}_V^{\otimes
i-1}\otimes \sigma\otimes \mathrm{id}_V^{\otimes n-i-1}\in
\mathrm{End}(V^{\otimes n})$. For any $w\in \mathfrak{S}_{n}$, we
denote by $T^\sigma_w$ the corresponding lift of $w$ in the braid
group $B_n$, defined as follows: if $w=s_{i_1}\cdots s_{i_l}$ is
any reduced expression of $w$, then $T^\sigma_w=\sigma_{i_1}\cdots
\sigma_{i_l}$. This definition is well-defined (see, e.g., Theorem
4.12 in \cite{KT}).

We define $\beta:T(V)\underline{\otimes} T(V)\rightarrow
T(V)\underline{\otimes} T(V)$ by requiring that the restriction of
$\beta$ on $V^{\otimes i}\underline{\otimes} V^{\otimes j}$,
denoted by $\beta_{ij}$, is $T^\sigma_{\chi_{ij}}$ , where
\[\chi_{ij}=\left(\begin{array}{cccccccc}
1&2&\cdots&i&i+1&i+2&\cdots & i+j\\
j+1&j+2&\cdots&j+i&1& 2 &\cdots & j
\end{array}\right)\in \mathfrak{S}_{i+j},\] for any $i,j\geq 1$. For convenience, we denote by
$\beta_{0i}$ and $\beta_{i0}$ the identity map of $V^{\otimes i}$.

Given an invertible element $q\in \mathbb{K}$ which is not a root
of unity, we denote $[0]_{q}=1$ and
$[n]_{q}=1+q+q^2+\cdots+q^{n-1}=\frac{1-q^n}{1-q}$ when $n\in
\mathbb{N}$. We also denote $[n]_{q}!=[1]_{q}\cdots [n]_{q}$.

\section{Braided algebras}
We start by recalling the notion of braided algebras which is the
relevant object of associative algebras in braided categories. In
the following, all algebras are always assumed to be associative,
but not necessarily unital.
\begin{definition}Let $A=(A,m)$ be an algebra with product $m$, and $\sigma$ be a braiding on $A$. We call the triple $(A,m,\sigma)$ a \emph{braided algebra} if it satisfies the following conditions:
\[
\begin{split}
(\mathrm{id}_A\otimes m)\sigma_1\sigma_2&=\sigma( m\otimes
\mathrm{id}_A),\\[3pt] ( m\otimes
\mathrm{id}_A)\sigma_2\sigma_1&=\sigma(\mathrm{id}_A\otimes m).
\end{split}
\]
Moreover, if $A$ is unital and its unit $1_A$ satisfies that for
any $a\in A$,
\[
\begin{split}
\sigma(a\otimes 1_A)&=1_A\otimes a,\\[3pt] \sigma(1_A\otimes a)&=a\otimes
1_A,
\end{split}
\]then $A$ is called a \emph{unital braided algebra}.
\end{definition}

Because all the constructions in this paper are based on braided
algebras, we provide several concrete examples which will either
be used in our later discussion or afford the reader some
illustrations. Some of them may be known, while some may be new.
For more examples, one can see \cite{B} and \cite{JR}.

\begin{example}Let $V$ be a vector space with basis $\{e_i\}$ which is at most countable. We provide a braided algebra structure on $V$. The braiding $\sigma$ on $V$ is given by $\sigma(e_i\otimes
e_j)=q_{ij}e_j\otimes e_i$, where $q_{ij}$'s are nonzero scalars
in $\mathbb{K}$ such that $q_{ij}q_{ik}=q_{i\ j+k}$ and
$q_{ik}q_{jk}=q_{i+j\ k}$ for any $i,j,k$. For instance, let $q$
be a nonzero scalar in $\mathbb{K}$ and $q_{ij}=q^{ij}$. A
multiplication $\cdot$ on $V$ which is compatible with the
braiding $\sigma$ is given as follows.

Case 1. If $V$ is a finite-dimensional vector space with basis
$\{e_1,e_2,\ldots,e_N\}$, then we define

\[e_i\cdot e_j= \left\{
\begin{array}{lll}
e_{i+j},&& \mathrm{if\ }i+j\leq N,\\[3pt]
0,&&\mathrm{otherwise}.
\end{array} \right.
\]

Case 2. If $V$ is a vector space with basis $\{e_i\}_{i\in
\mathbb{N}}$, then we define $e_i\cdot e_j=e_{i+j}$ for any
$i,j\in \mathbb{N}$.

It is evident that $\cdot$ is an associative algebra structure on
$V$ in both cases. Notice that
\begin{eqnarray*}(\mathrm{id}_V\otimes
\cdot)\sigma_1\sigma_2 (e_i\otimes e_j\otimes
e_k)&=&q_{jk}q_{ik}e_k\otimes e_{i+j}\\[3pt]
&=&q_{i+j\ k}e_k\otimes e_{i+j}\\[3pt]&=&\sigma(
\cdot\otimes\mathrm{id}_V)(e_i\otimes e_j\otimes
e_k),\end{eqnarray*}
 and similarly
$(
\cdot\otimes\mathrm{id}_V)\sigma_2\sigma_1=\sigma(\mathrm{id}_V\otimes
\cdot)$. Therefore $(V,\cdot,\sigma)$ is a braided algebra.

In particular, the polynomial algebra $\mathbb{K}[t]$ is a braided
algebra with respect to the braiding defined by $\sigma(t^n\otimes
t^m)=q^{nm}t^m\otimes t^n$.
\end{example}

\begin{example}All notions of this example can be found in \cite{Ka}. Let $q\neq 1$ be an invertible scalar in $\mathbb{K}$, and $x,y$ be two indeterminates. Denote by $\mathbb{K}_q[x,y]$ the quantum plane, i.e., the algebra generated by $x,y$ subject to the relation $yx=qxy$. It has a linear basis $\{x^iy^j\}_{i,j\geq 0}$. Define two algebra automorphisms $\omega_x$ and $\omega_y$ of $\mathbb{K}_q[x,y]$ by requiring that
$$\omega_x(x)=qx,\omega_x(y)=y,\omega_y(x)=x,\omega_y(y)=qy,$$ and define two endomorphisms $\partial_q/\partial x$ and $\partial_q/\partial y$ by requiring that $$\frac{\partial_q(x^my^n)}{\partial x}=[m]x^{m-1}y^n,\frac{\partial_q(x^my^n)}{\partial y}=[n]x^my^{n-1},$$ where $[k]=\frac{q^k-q^{-k}}{q-q^{-1}}$ for any $k\in \mathbb{N}$.

Let $U_q \mathfrak{sl}_2$ be the quantized algebra associated to
$\mathfrak{sl}_2$, i.e., the algebra generated by $E,F,K,K^{-1}$
subject to the relations $$KK^{-1}=K^{-1}K=1,$$
$$KE=q^2EK, \ KF=q^{-2}FK,$$  $$EF-FE=\frac{K-K^{-1}}{q-q^{-1}}.$$ It is well-known that $U_q \mathfrak{sl}_2$ is a quasi-triangular Hopf algebra.

By Theorem VII 3.3 in \cite{Ka}, $\mathbb{K}_q[x,y]$ is a $U_q
\mathfrak{sl}_2$-module-algebra with the following module
structure: for any $P\in \mathbb{K}_q[x,y] $,
$$EP=x\frac{\partial_q(P)}{\partial y}, FP=\frac{\partial_q(P)}{\partial x}y, $$ $$KP=(\omega_x\omega_y^{-1})(P), K^{-1}P=(\omega_y\omega_x^{-1})(P).$$
Set $V=\mathrm{Span}_\mathbb{K}\{x,y\}$. It is not hard to see
that the above action restricting on $V$ is the standard
2-dimensional simple $U_q \mathfrak{sl}_2$-module structure. We
know that (Theorem 2.7 in \cite{JR}) every module-algebra over a
quasi-triangular Hopf algebra has a braided algebra structure. So
$\mathbb{K}_q[x,y]$ is a braided algebra.\end{example}

\begin{example}Let $(V,\sigma)$ be a braided vector space.  It
is know that $(T(V),m,\beta)$ is a braided algebra, where $m$ is
the concatenation product. Let $M_n: V^{\otimes n}\rightarrow
T(V)$ be a linear map such that $\beta(M_n\otimes
\mathrm{id}_V)=(\mathrm{id}_V\otimes M_n)\beta_{n1}$ and
$\beta(\mathrm{id}_V\otimes M_n)=(M_n\otimes\mathrm{id}_V
)\beta_{1n}$. If we denote by $\mathcal{I}$ the ideal of $T(V)$
generated by $\mathrm{Im}M_n$, the image of $M_n$, then
$\beta\big(T(V)\underline{\otimes}\mathcal{I}+\mathcal{I}\underline{\otimes}T(V)\big)\subset
T(V)\underline{\otimes}\mathcal{I}+\mathcal{I}\underline{\otimes}T(V)$.
So the quotient algebra $T(V)/\mathcal{I}$ is also a braided
algebra. For instance, if $M_2=\mathrm{id}_V^{\otimes 2}-\sigma$,
then the quotient algebra is the $r$-symmetric algebra defined in
\cite{B}.
\end{example}

\begin{example}Let $H$ be a finite dimensional quasi-triangular Hopf algebra. By a result of Majid (Theorem 3.3 in \cite{M}), the quantum double $\mathcal{D}(H)$ of $H$ is a braided algebra (according to a discussion in \cite{JR} for Radford's work \cite{Ra}).\end{example}

\section{Quantum quasi-shuffle algebras}
For any algebra $A$, it is a braided algebra with respect to the
flip map switching the two factors of $A\otimes A$. One can
construct an algebra structure on $T(A)$ which combines the
multiplication of $A$ and the shuffle product of $T(A)$ (see
\cite{NR}). This structure is the so-called quasi-shuffle algebra.
If the flip map is replaced by a general braiding, one can
construct a quantized quasi-shuffle product by assuming some
compatibilities between the multiplication of $A$ and the braiding
(for more details, one can see \cite{JR} and \cite{JRZ}). More
precisely, given a braided algebra $(A, m,\sigma)$, the
\emph{quantum quasi-shuffle product} $\Join_\sigma$ on $T(A)$ is
given by the following formulas.

For any $\lambda\in \mathbb{K}$ and $x\in T(A)$,
$$\lambda\Join_\sigma x=x \Join_\sigma \lambda=\lambda \cdot x.$$

For $i,j\geq 2$ and any $a_1,\ldots, a_i,b_1,\ldots, b_j\in A$,
$\Join_\sigma$ is defined recursively by
$$a_1\Join_\sigma b_1= a_1\otimes b_1+\sigma(a_1\otimes
b_1)+m(a_1\otimes b_1),$$
\begin{eqnarray*}\lefteqn{ a_1\Join_\sigma  (b_1\otimes
\cdots\otimes
b_j)}\\[3pt]
&=&a_1\otimes b_1\otimes \cdots\otimes  b_{j}+(\mathrm{id}_A\otimes\Join_{\sigma  (1,j-1)})(\beta_{1,1}\otimes\mathrm{id}_A^{\otimes  j-1})(a_1\otimes b_1\otimes \cdots\otimes  b_j)\\[3pt]
&&+m(a_1\otimes  b_1)\otimes b_2\otimes \cdots\otimes  b_j,
\end{eqnarray*}
\begin{eqnarray*}
\lefteqn{(a_1\otimes \cdots\otimes  a_i)\Join_\sigma  b_1}\\[3pt]
&=&a_1\otimes  \big((a_2\otimes \cdots\otimes  a_i)\Join_\sigma  b_1\big)+\beta_{i,1}(a_1\otimes \cdots\otimes  a_i\otimes  b_1)\\[3pt]
&&+(m\otimes\mathrm{id}_A^{\otimes  i-1})(\mathrm{id}_A\otimes
\beta_{i-1,1})(a_1\otimes \cdots\otimes a_i\otimes b_1),
\end{eqnarray*}
and\begin{eqnarray*}
\lefteqn{(a_1\otimes \cdots\otimes  a_i)\Join_\sigma  (b_1\otimes \cdots\otimes  b_j)}\\[3pt]
&=&a_1\otimes  \big((a_2\otimes \cdots\otimes  a_i)\Join_\sigma (b_1\otimes \cdots\otimes  b_{j})\big)\\[3pt]
&&+(\mathrm{id}_A\otimes  \Join_{\sigma  (i,j-1)})(\beta_{i,1}\otimes \mathrm{id}_A^{\otimes  j-1})(a_1\otimes \cdots\otimes  a_i\otimes  b_1\otimes \cdots\otimes  b_j)\\[3pt]
&&+(m\otimes \Join_{\sigma  (i-1,j-1)} )(\mathrm{id}_A\otimes
\beta_{i-1,1}\otimes \mathrm{id}_A^{\otimes  j-1})(a_1\otimes
\cdots\otimes a_i\otimes b_1\otimes \cdots\otimes  b_j),
\end{eqnarray*}
where $\Join_{\sigma (i,j)}$ denotes the restriction of
$\Join_\sigma $ on $A^{\otimes  i}\underline{\otimes} A^{\otimes
j}$.

\begin{remark}1. Given a braided algebra $(A, m,\sigma)$, $T_{\sigma}^{qsh}(A)=(T(A),\Join_\sigma)$ is a an associative algebra with unit $1\in \mathbb{K}$, and called the \emph{quantum quasi-shuffle algebra} built on $(A,
m,\sigma)$. Furthermore, the algebra $T_{\sigma}^{qsh}(A)$,
together with the braiding $\beta$ and the deconcatenation
coproduct $\delta$, forms a braided bialgebra in the sense of
\cite{Ta} (see \cite{JR}). The set of primitive elements is
exactly $A$.

2. By using Example 2.2, Proposition 17 in \cite{JRZ} can be
applied to any vector space whose basis is at most countable. In
other words, for any vector space $V$ with at most countable
basis, one can provide a linear basis of $T(V)$ by combining the
quantum quasi-shuffle product with Lyndon words.\end{remark}

\begin{example}[Hoffman's q-deformation]In \cite{Hof}, Hoffman defined his q-deformation of
quasi-shuffles. It is an attempt to deform the quasi-shuffle
product according to the quantum shuffle product. Now we give an
explanation of the q-deformation from a point of view of quantum
quasi-shuffles. Let $X$ be a locally finite set, i.e., $X$ is a
disjoint union of finite set $X_n$, whose elements are called
letters of degree n, for $n\geq 1$. We denote by $\mathfrak{X}$
the vector space spanned by $X$. The elements in
$T(\mathfrak{X})$, which are of the form $a_1\otimes a_2\otimes
\cdots \otimes a_m$ with $a_i\in X$, are called words. Let $[ ,]$
be a graded associative product on $\mathfrak{X}$. Hoffman defined
an associative product $\ast_q$ on $T(\mathfrak{X})$: for any
words $w_1,w_2\in T(\mathfrak{X})$ and letters $a,b\in X $,
\begin{eqnarray*}(a\otimes
w_1)\ast_q (b\otimes w_2) &=&a\otimes \big(w_1\ast_q(b\otimes
w_2)\big)+q^{|a\otimes w_1||b|}b\otimes \big((a\otimes w_1)\ast_q
w_2
\big)\\[3pt]
&&+q^{|w_1||b|}[a,b]\otimes(w_1\ast_q w_2),\end{eqnarray*}where
$|w|$ is the degree of a word, i.e., the sum of degrees of its
factors.

We define a braiding $\sigma$ on $\mathfrak{X}$ as follows: for
$x\in X_i$ and $y\in X_j$, $\sigma(x\otimes y)=q^{ij}y\otimes x$,
where $q\in K$ is a nonzero scalar. Since the product $[ ,]$
preserves the grading, it is an easy exercise to verify that
$(\mathfrak{X},[,],\sigma)$ is a braided algebra. By comparing
their reductive formulas, one can see that the quantum
quasi-shuffle algebra built on $(\mathfrak{X},[,],\sigma)$ is just
Hoffman's q-deformation.\end{example}

A similar discussion can be given to the formula of the
multiplication rule of quantum quasi-monomial functions in
\cite{TU}.

\section{An explicit formula for the quantum quasi-shuffle product}
In order to give a more explicit description of the quantum
quasi-shuffle product, we need to recall some terminologies
introduced in \cite{GK}. An $(i,j)$-\emph{shuffle} is an element
$w\in \mathfrak{S}_{i+j}$ such that $w (1) < \cdots <w (i)$ and $w
(i+1) < \cdots <w (i+j)$. We denote by $\mathfrak{S}_{i,j}$ the
set of all $(i,j)$-shuffles. Given an $(i,j)$-shuffle $w$, a pair
$(k,k+1)$, where $1\leq k < i+j$, is called an \emph{admissible
pair} for $w$ if $w^{-1}(k)\leq i<w^{-1}(k+1)$. We denote by
$\mathcal{T}^w$ the set of all admissible pairs for $w$. For any
subset $S$ of $\mathcal{T}^w$, the pair $(w,S)$ is called a
\emph{mixable $(i,j)$-shuffle}. We denote by
$\overline{\mathfrak{S}}_{i,j}$ the set of all mixable
$(i,j)$-shuffles, i.e.,
$$\overline{\mathfrak{S}}_{i,j}=\{(w,S)|w\in \mathfrak{S}_{i,j}, S\subset \mathcal{T}^w\}.$$

Let $(A, m,\sigma)$ be a braided algebra. Define $m^k:A^{\otimes
k+1}\rightarrow A$ recursively by $m^0=\mathrm{id}_A$, $m^1=m$ and
$m^k=m(\mathrm{id}_A\otimes m^{k-1})$ for $k\geq 2$. Given
$n\in\mathbb{N}$, we denote
$$\mathcal{C}(n)=\{I=(i_1,\ldots,i_k)\in
\mathbb{N}^k|1\leq k\leq n, i_1+\cdots+i_k=n\}.$$
 The elements in $\mathcal{C}(n)$ are called \emph{compositions} of
 $n$. For any $I=(i_1,\ldots,i_k)\in \mathcal{C}(n)$, we define $m_I=m^{i_1-1}\otimes\cdots\otimes
 m^{i_k-1}$. For any $(w,S)\in\overline{\mathfrak{S}}_{i,j}$, we
 associate to $S$ a composition $\textrm{cp}(S)$ of $i+j$ as follows: if
$S=\{(k_1,k_1+1),\ldots, (k_s,k_s+1)\}$ with $k_1<\cdots <k_s$,
set  $$\textrm{cp}(S)=(\underbrace{1,\ldots\ldots,1}_{k_1-1\
\textrm{copies}},
2,\underbrace{1,\ldots\ldots\ldots\ldots,1}_{k_2-k_1-2\
\textrm{copies}},2,\ldots,2,\underbrace{1,\ldots\ldots\ldots,1}_{i+j-k_s-1\
\textrm{copies}}).$$  By convention, we set
$\textrm{cp}(\emptyset)=(1,1,\ldots,1)$. Denote
$T_{(w,S)}^\sigma=m_{\textrm{cp}(S)}\circ T_w^\sigma$. For instance, let \[ w=\left(\begin{array}{ccccccccc}
1&2&3&4&5&6&7&8&9\\
2&4&6&8&1&3&5&7&9
\end{array}\right)\in\mathfrak{S}_{4,5},
\]and $S=\{(2,3),(4,5),(8,9)\}$. Then $\textrm{cp}(S)=(1,2,2,1,1,2)$ and $$T_{(w,S)}^\sigma=(\mathrm{id}_A\otimes m\otimes m\otimes \mathrm{id}_A\otimes \mathrm{id}_A\otimes m)\sigma_1\sigma_3\sigma_5\sigma_7\sigma_2\sigma_4\sigma_6\sigma_3\sigma_5\sigma_4.$$

\begin{theorem}Let $(A, m,\sigma)$ be a braided algebra. Then for any $a_1,\ldots,a_{i+j}\in A$,  $$(a_1\otimes\cdots\otimes a_i)\Join_\sigma (a_{i+1}\otimes\cdots\otimes a_{i+j})=\sum_{(w,S)\in\overline{\mathfrak{S}}_{i,j}}T_{(w,S)}^\sigma(a_1\otimes \cdots\otimes a_{i+j}).$$\end{theorem}
\begin{proof}We use induction on $i+j$.

When $i=j=1$,
$$a_1\Join_\sigma a_2=m(a_1\otimes a_2)+a_1\otimes a_2+\sigma(a_1\otimes a_2).$$
On the other hand,
$\overline{\mathfrak{S}}_{1,1}=\{(1_{\mathfrak{S}_2},\emptyset),
(1_{\mathfrak{S}_2},\{(1,2)\}), (s_1,\emptyset)\}$, where $s_1$ is
the generator of $\mathfrak{S}_2$. So the formula holds.

We assume that the formula is true in the case $\leq i+j$. By the
inductive formula of quantum quasi-shuffles, we have that
\begin{eqnarray*}\lefteqn{(a_1\otimes\cdots\otimes a_{i+1})\Join_\sigma (a_{i+2}\otimes\cdots\otimes a_{i+j+1})}\\[3pt]
&=&a_1\otimes \big((a_2\otimes\cdots\otimes a_{i+1})\Join_\sigma (a_{i+2}\otimes\cdots\otimes a_{i+j+1})\big)\\
&&+(\mathrm{id}_A\otimes \Join_{\sigma (i+1,j-1)})(\beta_{i+1,1}\otimes \mathrm{id}_A^{\otimes j-1})(a_1\otimes \cdots\otimes a_{i+j+1})\nonumber\\
&&+(m\otimes\Join_{\sigma (i,j-1)} )(\mathrm{id}_A\otimes
\beta_{i,1}\otimes \mathrm{id}_A^{\otimes
j-1})(a_1\otimes \cdots\otimes a_{i+j+1})\\[3pt]
&=&\sum_{(w,S)\in\overline{\mathfrak{S}}_{i,j}}(\mathrm{id}_A\otimes T_{(w,S)}^\sigma)(a_1\otimes \cdots\otimes a_{i+j+1})\\
&&+\sum_{(w,S)\in\overline{\mathfrak{S}}_{i+1,j-1}}(\mathrm{id}_A\otimes T_{(w,S)}^\sigma)(\beta_{i+1,1}\otimes \mathrm{id}_A^{\otimes j-1})(a_1\otimes \cdots\otimes a_{i+j+1})\nonumber\\
&&+\sum_{(w,S)\in\overline{\mathfrak{S}}_{i,j-1}}(m\otimes
T_{(w,S)}^\sigma )(\mathrm{id}_A\otimes \beta_{i,1}\otimes
\mathrm{id}_A^{\otimes j-1})(a_1\otimes \cdots\otimes
a_{i+j+1})\\[3pt]
&=&\sum_{(w,S)\in\overline{\mathfrak{S}}_{i,j}}(\mathrm{id}_A\otimes m_{\textrm{cp}(S)})T_{1_{\mathfrak{S}_1}\times w}^\sigma(a_1\otimes \cdots\otimes a_{i+j+1})\\
&&+\sum_{(w,S)\in\overline{\mathfrak{S}}_{i+1,j-1}}(\mathrm{id}_A\otimes m_{\textrm{cp}(S)})T_{(1_{\mathfrak{S}_1}\times w)\circ (\chi_{i+1,1}\times 1_{\mathfrak{S}_{j-1}})}^\sigma(a_1\otimes \cdots\otimes a_{i+j+1})\nonumber\\
&&+\sum_{(w,S)\in\overline{\mathfrak{S}}_{i,j-1}}(m\otimes
m_{\textrm{cp}(S)} )T^\sigma_{(1_{\mathfrak{S}_2}\times w)\circ
(1_{\mathfrak{S}_1}\times \chi_{i,1}\times
1_{\mathfrak{S}_{j-1}})}(a_1\otimes \cdots\otimes
a_{i+j+1}),\end{eqnarray*} where the third equality follows from
the fact that all the expressions of the permutations being lifted
are reduced.

Denote $$\mathcal{S}_1=\{(w,S)\in
\overline{\mathfrak{S}}_{i+1,j}|(1,2)\notin S, w(1)=1\},$$
$$\mathcal{S}_2=\{(w,S)\in \overline{\mathfrak{S}}_{i+1,j}|(1,2)\notin S, w(i+2)=1\},$$and
$$\mathcal{S}_3=\{(w,S)\in \overline{\mathfrak{S}}_{i+1,j}|(1,2)\in S\}.$$For any $(i+1,j)$-shuffle $w$, one has either $w(1)=1$ or $w(i+2)=1$.
Therefore $\mathcal{S}_1$, $\mathcal{S}_2$ and $ \mathcal{S}_3$
are mutually disjoint,  and
$\mathfrak{S}_{i+1,j}=\mathcal{S}_1\cup\mathcal{S}_2\cup\mathcal{S}_3$.

We make a further observation. It is easy to see that there is a
one-to-one correspondence between $\mathfrak{S}_{i,j}$ and $\{w\in
\mathfrak{S}_{i+1,j}|w(1)=1\}$ given by $w\mapsto
1_{\mathfrak{S}_1}\times w$ for any $w\in \mathfrak{S}_{i,j}$. So
$$\mathcal{S}_1=\{(1_{\mathfrak{S}_1}\times w,S)|w\in \mathfrak{S}_{i,j}, S\subset\mathcal{T}^{1_{\mathfrak{S}_1}\times w},(1,2)\notin S\}.$$
There is a one-to-one correspondence between
$\mathfrak{S}_{i+1,j-1}$ and $\{w\in
\mathfrak{S}_{i+1,j}|w(i+2)=1\}$ given by  $w\mapsto
\widetilde{w}=(1_{\mathfrak{S}_1}\times w)\circ
(\chi_{i+1,1}\times 1_{\mathfrak{S}_{j-1}})$ for any $w\in
\mathfrak{S}_{i+1,j-1}$. Consequently,
$$\mathcal{S}_2=\{(\widetilde{w},S)|w\in \mathfrak{S}_{i+1,j-1},S\subset\mathcal{T}^{\widetilde{w}},(1,2)\notin S\}.$$
Finally, for any $(w,S)\in \mathcal{S}_3$, we must have that
$w(1)=1$ and $w(i+2)=2$. There is a one-to-one correspondence
between $\mathfrak{S}_{i,j-1}$ and $\{w\in
\mathfrak{S}_{i+1,j}|w(1)=1,w(i+2)=2\}$ given by $w\mapsto
\overline{w}=(1_{\mathfrak{S}_2}\times w)\circ
(1_{\mathfrak{S}_1}\times \chi_{i,1}\times
1_{\mathfrak{S}_{j-1}})$ for any $w\in \mathfrak{S}_{i,j-1}$. So
$$\mathcal{S}_3=\{(\overline{w},S)\in \overline{\mathfrak{S}}_{i+1,j}|w\in \mathfrak{S}_{i,j-1},(1,2)\in S\}.$$

 As a conclusion, the three terms in the final step of the preceding
 computation come from $\mathcal{S}_1$, $\mathcal{S}_2$ and
 $\mathcal{S}_3$ respectively. So we have that $$(a_1\otimes\cdots\otimes a_{i+1})\Join_\sigma
(a_{i+2}\otimes\cdots\otimes
v_{i+j+1})=\sum_{(w,S)\in\overline{\mathfrak{S}}_{i+1,j}}T_{(w,S)}^\sigma(a_1\otimes
\cdots\otimes a_{i+j+1}),$$ which completes the induction.
 \end{proof}

\begin{remark}Let $(A,m)$ be an algebra and $\lambda$ be a scalar in $\mathbb{K}$. Then $(A, \lambda m)$ becomes a braided algebra with respect to the usual flip map. In this case, the formula in Theorem 4.1 coincides with the one of mixable shuffle product introduced in \cite{GK}.\end{remark}

\section{The Subalgebra generated by primitive elements}

Assume again that $(A,m,\sigma)$ is a braided algebra. We denote
by $S^{qsh}_\sigma(A)$ the subalgebra of $T^{qsh}_\sigma(A)$
generated by $A$. To describe this subalgebra, we need to
introduce some notation.

For a fixed $n\in \mathbb{N}$ and any $w\in \mathfrak{S}_n$, we
denote $$\mathcal{S}^w=\{(k,k+1)|1\leq k<n,
w^{-1}(k)<w^{-1}(k+1)\},$$and
$$\overline{\mathfrak{S}}_n=\{(w,S)|w\in \mathfrak{S}_n, S\subset
\mathcal{S}^w\}.$$For any $(w,S)\in \overline{\mathfrak{S}}_n$, we
associate to $S$ a composition $\textrm{cp}(S) $ of $n$ as
follows. Let $S=\{(k_1,k_1+1), \ldots, (k_s,k_s+1)\}$ with
$k_1<\cdots<k_s$. We divide $\{k_1,\ldots,k_s\}$ into several
subsets
$$\{k_1,\ldots,k_{i_1}\}, \{k_{i_1+1},\ldots,k_{i_1+i_2}\},
\ldots,\{k_{i_1+\cdots+i_{r-1}+1},\ldots,k_s\},$$ which obey the
rule that:
\[\left\{
\begin{array}{lll}
k_1+1=k_2,k_2+1=k_3,\ldots,k_{i_1-1}+1=k_{i_1},&&\\
k_{i_1+1}+1=k_{i_1+2},k_{i_1+2}+1=k_{i_1+3},\ldots,k_{i_1+i_2-1}+1=k_{i_2},&&\\
\ldots,&&\\
k_{i_1+\cdots+i_{r-1}+1}+1=k_{i_1+\cdots+i_{r-1}+2},\ldots,
k_{s-1}+1=k_{s},&&
\end{array} \right.
\]
 but $k_{i_1}+1<k_{i_1+1},
k_{i_1+i_2}+1<k_{i_1+i_2+1},\ldots,
k_{i_1+\cdots+i_{r-1}}+1<k_{i_1+\cdots+i_{r-1}+1}$. Denote
$i_r=s-i_1-\cdots-i_{r-1}$. Then we write
 $$\textrm{cp}(S)=(\underbrace{1,\ldots\ldots,1}_{k_1-1\ \textrm{copies}}, i_1+1,\underbrace{1,\ldots\ldots\ldots\ldots,1}_{k_{i_1+1}-k_{i_1}-1\ \textrm{copies}},i_2+1,\ldots,i_r+1,\underbrace{1,\ldots\ldots\ldots,1}_{n-k_{i_r}-1\ \textrm{copies}}).$$For instance, let \[ w=\left(\begin{array}{ccccccccc}
1&2&3&4&5&6&7&8&9\\
1&6&5&4&8&7&2&9&3
\end{array}\right)\in\mathfrak{S}_9,
\]and $S=\{(1,2),(2,3),(6,7),(8,9)\}$. We divide $\{1,2,6,8\}$ into subsets $\{1,2\},\{6\},\{8\}$. Then  $\textrm{cp}(S)=(3,1,1,2,2)$.
We define as before the map
$T_{(w,S)}^\sigma=m_{\textrm{cp}(S)}\circ T_w^\sigma$ for any
$(w,S)\in\overline{\mathfrak{S}}_n$.

Now we provide a decomposition of $\overline{\mathfrak{S}}_{n+1}$
which will be used later. For any $1\leq i\leq n+1$, we denote
$\mathfrak{S}_{n+1}(i)=\{w\in \mathfrak{S}_{n+1}|w(1)=i\}$. It is
clear that $\mathfrak{S}_{n+1}$ is the disjoint union of all
$\mathfrak{S}_{n+1}(i)$'s, and for each $i$ there is a one-to-one
correspondence between $\mathfrak{S}_{n}$ and
$\mathfrak{S}_{n+1}(i)$ given by $w\mapsto
L(w,i)=(\chi_{1,i-1}\times
1_{\mathfrak{S}_{n+1-i}})\circ(1_{\mathfrak{S}_1}\times w) $ for
any $w\in \mathfrak{S}_{n-1}$. So
\begin{eqnarray*}\mathfrak{S}_{n+1}&=&\bigcup_{i=1}^{n+1}\mathfrak{S}_{n+1}(i)\\[3pt]
&=&\bigcup_{i=1}^{n+1}\bigcup_{w\in
\mathfrak{S}_{n}}\{L(w,i)\}\\[3pt]
&=&\bigcup_{w\in
\mathfrak{S}_{n}}\bigcup_{i=1}^{n+1}\{L(w,i)\}.\end{eqnarray*}
Then we have that
\begin{eqnarray*}\overline{\mathfrak{S}}_{n+1}&=&\bigcup_{w\in
\mathfrak{S}_{n}}\bigcup_{i=1}^{n+1}\{(L(w,i),S)|S\subset \mathcal{S}^{L(w,i)}\}\\[3pt]
&=&\big(\bigcup_{w\in
\mathfrak{S}_{n}}\bigcup_{i=1}^{n+1}\{(L(w,i),S)|S\subset
\mathcal{S}^{L(w,i)},(i,i+1)\notin
S\}\big)\\[3pt]
&&\cup\big(\bigcup_{w\in
\mathfrak{S}_{n}}\bigcup_{i=1}^{n+1}\{(L(w,i),S)|S\subset
\mathcal{S}^{L(w,i)},(i,i+1)\in S\}\big).\end{eqnarray*} All the
unions above are disjoint.

Given $w\in \mathfrak{S}_n$ and $S\subset \mathcal{S}^w$ with
$\mathrm{cp}(S)=(i_1,\ldots,i_s)$. For any $0\leq k\leq s$, we
denote
$$\mathrm{cp}(S)_k=(i_1,\ldots,i_k,1,i_{k+1},\ldots,i_s),$$
$$\mathrm{cp}(S)^k=(i_1,\ldots,i_{k-1},i_k+1,i_{k+1},\ldots,i_s),$$
and $$I_k=(\underbrace{1,\ldots,1}_{k\
\textrm{copies}},2,\underbrace{1,\ldots\ldots\ldots,1}_{n-1-|S|-k\
\textrm{copies}}).$$Here, $|S|$ denotes the cardinality of the set
$S$.

\begin{lemma}Under the assumptions above, we have \[\begin{split}
(\beta_{1,k}\otimes \mathrm{id}_A^{\otimes
n-|S|-k})(\mathrm{id}_A\otimes
T_{(w,S)}^\sigma)&=m_{\mathrm{cp}(S)_k}T_{L(w,i_1+\cdots+i_k+1)},\\[5pt]
m_{I_k}(\beta_{1,k}\otimes \mathrm{id}_A^{\otimes
n-|S|-k})(\mathrm{id}_A\otimes
T_{(w,S)}^\sigma)&=m_{\mathrm{cp}(S)^k}T_{L(w,i_1+\cdots+i_k+1)}.
\end{split}
\]\end{lemma}
\begin{proof}By an easy induction (or see the proof of Prop. 4.21 in \cite{JR}), one can show that $\sigma(\mathrm{id}_A\otimes m^l)=(m^l\otimes
\mathrm{id}_A)\beta_{1,l+1}$ for any $l$. So we have \begin{eqnarray*}\lefteqn{(\beta_{1,k}\otimes
\mathrm{id}_A^{\otimes n-|S|-k})(\mathrm{id}_A\otimes
T_{(w,S)}^\sigma)}\\[3pt]
&=&(\beta_{1,k}\otimes \mathrm{id}_A^{\otimes
n-|S|-k})(\mathrm{id}_A\otimes
m_{\textrm{cp}(S)})(\mathrm{id}_A\otimes T_w^\sigma)\\[3pt]
&=&m_{\textrm{cp}(S)_k}(\beta_{1,i_1+\cdots+i_k}\otimes
\mathrm{id}_A^{\otimes i_{k+1}+\cdots+i_s})(\mathrm{id}_A\otimes T_w^\sigma)\\[3pt]
&=&m_{\textrm{cp}(S)_k}T_{(\chi_{1,i_1+\cdots+i_k}\times
1_{\mathfrak{S}_{i_{k+1}+\cdots+i_s}})\circ(1_{\mathfrak{S}_1}\times
w)}^\sigma\\[3pt]
&=&m_{\mathrm{cp}(S)_k}T_{L(w,i_1+\cdots+i_k+1)}.\end{eqnarray*}

The second equality is a consequence of the first one.\end{proof}

\begin{theorem}Let $(A,m,\sigma)$ be a braided algebra. For any $a_1,\ldots,a_n\in A$, we have that$$a_1\Join_\sigma\cdots\Join_\sigma a_n=\sum_{(w,S)\in\overline{\mathfrak{S}}_n}T_{(w,S)}^\sigma(a_1\otimes \cdots\otimes a_{n}).$$ Therefore $S^{qsh}_\sigma(A)=\sum_{n\geq 0}\mathrm{Im}(\sum_{(w,S)\in\overline{\mathfrak{S}}_n}T_{(w,S)}^\sigma)$.\end{theorem}
\begin{proof}We use induction on $n$.

When $n=2$, it is trivial since
$$\overline{\mathfrak{S}}_2=\{(1_{\mathfrak{S}_2},\emptyset),(1_{\mathfrak{S}_2},\{(1,2)\}),(s_1,\emptyset)
\}.$$ By Theorem 3.5, we have that for any $a_1,\ldots,a_{r+1}\in
A$,
\begin{eqnarray*}\lefteqn{a_1\Join_\sigma (a_2\otimes\cdots\otimes a_{r+1})}\\[3pt]
&=&\sum_{k=0}^r(\beta_{1,k}\otimes \mathrm{id}_A^{\otimes r-k})(a_1\otimes \cdots\otimes a_{r+1})\\[3pt]
&&+\sum_{k=0}^{r-1}(\mathrm{id}_A^{\otimes k}\otimes m\otimes
\mathrm{id}_A^{\otimes r-k-1})(\beta_{1,k}\otimes
\mathrm{id}_A^{\otimes r-k})(a_1\otimes \cdots\otimes
a_{r+1}).\end{eqnarray*} Therefore,
\begin{eqnarray*}\lefteqn{a_1\Join_\sigma\cdots\Join_\sigma
a_{n+1}}\\[3pt]
&=&a_1\Join_\sigma\big(\sum_{(w,S)\in\overline{\mathfrak{S}}_n}T_{(w,S)}^\sigma(a_2\otimes
\cdots\otimes a_{n+1})\big)\\[3pt]
&=&\sum_{(w,S)\in\overline{\mathfrak{S}}_n}\sum_{k=0}^{n-|S|}(\beta_{1,k}\otimes
\mathrm{id}_A^{\otimes n-|S|-k})\big(a_1\otimes
T_{(w,S)}^\sigma(a_2\otimes
\cdots\otimes a_{n+1})\big)\\[3pt]
&&+\sum_{(w,S)\in\overline{\mathfrak{S}}_n}\sum_{k=0}^{n-|S|-1}(\mathrm{id}_A^{\otimes
k}\otimes m\otimes \mathrm{id}_A^{\otimes
n-|S|-k-1})\\[3pt]
&&\ \ \ \ \ \ \ \ \ \ \ \ \ \ \ \ \ \ \  \ \ \ \ \ \ \ \ \
\circ(\beta_{1,k}\otimes \mathrm{id}_A^{\otimes
n-|S|-k})\big(a_1\otimes T_{(w,S)}^\sigma(a_2\otimes \cdots\otimes
a_{n+1})\big)\\[3pt]
&=&\sum_{(w,S)\in\overline{\mathfrak{S}}_n}\sum_{k=0}^{n-|S|}m_{\textrm{cp}(S)_k}T_{L(w,i_1+\cdots+i_k+1)}^\sigma(a_1\otimes
\cdots\otimes a_{n+1})\\[3pt]
&&+
\sum_{(w,S)\in\overline{\mathfrak{S}}_n}\sum_{k=0}^{n-|S|-1}m_{\textrm{cp}(S)^k}T_{L(w,i_1+\cdots+i_k+1)}^\sigma(a_1\otimes
\cdots\otimes a_{n+1}),\end{eqnarray*} where the last equality
follows from the preceding lemma.

On the other hand, by the decomposition of
$\overline{\mathfrak{S}}_{n+1}$ mentioned before,
\begin{eqnarray*}\lefteqn{\sum_{(w,S)\in\overline{\mathfrak{S}}_{n+1}}T_{(w,S)}^\sigma(a_1\otimes
\cdots\otimes a_{n+1})}\\[3pt]
&=&\sum_{w\in\mathfrak{S}_n}\sum_{i=1}^{n+1}\sum_{\substack{S\subset
\mathcal{S}^{L(w,i)}\\(i,i+1)\notin S}}T_{(L(w,i),S)}^\sigma
(a_1\otimes
\cdots\otimes a_{n+1})\\[3pt]
&&+\sum_{w\in\mathfrak{S}_n}\sum_{i=1}^{n+1}\sum_{\substack{S\subset
\mathcal{S}^{L(w,i)}\\(i,i+1)\in
S}}T_{(L(w,i),S)}^\sigma(a_1\otimes \cdots\otimes
a_{n+1}).\end{eqnarray*}We compare the terms in these two
expressions. Notice that for a fixed $1\leq i\leq n+1$ and
$S\subset \mathcal{S}^{L(w,i)}$ with $(i,i+1)\notin S$, there is a
unique $S'\in \mathcal{S}^w$ such that
$\textrm{cp}(S)=\textrm{cp}(S')_i$. Indeed, we can write down
$S^\prime$ explicitly: if $S=\{(k_1,k_1+1), \ldots, (k_s,k_s+1)\}$
with $k_1<\cdots<k_l<i<k_{l+1}<\cdots<k_s$, then
$S'=\{(k_1,k_1+1), \ldots,
(k_l,k_l+1),(k_{l+1}-1,k_{l+1}),\ldots,(k_s-1,k_s)\}$. Similarly,
if $(i,i+1)\in S$, there is a unique $S''\in \mathcal{S}^w$ such
that $\textrm{cp}(S)=\textrm{cp}(S'')^i$. It follows that every
term in
$\sum_{(w,S)\in\overline{\mathfrak{S}}_{n+1}}T_{(w,S)}^\sigma(a_1\otimes
\cdots\otimes a_{n+1})$ is from exactly one term in the formula of
$a_1\Join_\sigma\cdots\Join_\sigma a_{n+1}$. The converse is also
true. Since all terms in each formula are mutually distinct, we
get the conclusion.
\end{proof}

\begin{remark}Consider Example 3.2 and let $(\mathfrak{X},[,],\sigma)$ be the braided algebra introduced there. For any $w\in \mathfrak{S}_n$, we denote $\iota(w)=\{(i,j)|1\leq i<j\leq n, w(i)>w(j)\}$. Then for any $a_1,\cdots, a_n\in X$, $$T_w^\sigma(a_1\otimes\cdots\otimes a_n)=q^{\sum_{(i,j)\in \iota(w)}|a_i||a_j|}a_{w^{-1}(1)}\otimes \cdots \otimes a_{w^{-1}(n)}.$$

For any two compositions $I=(i_1,\ldots,i_k)$ and
$J=(j_1,\ldots,j_l)$ of $n$, we say $I$ is a \emph{refinement} of
$J$, written by $I\succeq J$, if there are $r_1,\ldots,r_l\in
\mathbb{N}$ such that $r_1+\cdots +r_l=k$ and
 $$i_1+\cdots+i_{r_1}=j_1, i_{r_1+1}+\cdots+i_{r_1+r_2}=j_2,\ldots, i_{r_1+\cdots+r_{l-1}+1}+\cdots+i_k=j_l.$$
 For instance, $(1,2,2,3)\succeq (3,2,3)$. For any $w\in \mathfrak{S}_n$, let $C(w)$ be the composition $(i_1,\ldots,i_k)$ of $n$ such that $$\{i_1,i_1+i_2,\ldots, i_1+\cdots+i_{k-1}\}=\{l|1\leq l\leq n-1, w(l)>w(l+1)\}.$$
For any $I=(i_1,\ldots,i_l)\in \mathcal{C}(n)$, we write
$I[a_1\otimes \cdots a_n]=[,]_I(a_1\otimes \cdots a_n)$.

Observing that for any $w\in \mathfrak{S}_n$ there is a one-to-one
correspondence between $S\subset \mathcal{S}^w$ and $I\in
\mathcal{C}(n)$ with $I\succeq C(w)$, one has immediately that
$$a_1\ast_q\cdots\ast_q
a_n=\sum_{w\in\mathfrak{S}_n}q^{\sum_{(i,j)\in
\iota(w)}|a_i||a_j|}\sum_{I\succeq C(w)}I[a_{w^{-1}(1)}\otimes
\cdots \otimes a_{w^{-1}(n)}].$$ This formula is given by Hoffman
when $[,]$ is commutative (see Lemma 5.2 in \cite{Hof}).
\end{remark}

We conclude this section by an interesting formula. Let $V$ be a
vector space with basis $\{e_i\}_{i\in \mathbb{N}}$, and
$(V,m,\sigma)$ be the braided algebra structure given in Example
2.3. So we have that $\sigma(e_i\otimes e_j)=q_{ij}e_j\otimes e_i$
and $m(e_i\otimes e_j)=e_{i+j}$.

\begin{proposition}For any $i,k\in \mathbb{N}$, we have that $$e_i^{\Join_\sigma k}=\sum_{n=1}^k\sum_{\substack{l_1+\cdots+l_n=k\\1\leq l_1,\ldots,l_n\leq k}}\frac{[k]_{q_{ii}}!}{[l_1]_{q_{ii}}!\cdots [l_n]_{q_{ii}}!}e_{l_1i}\otimes \cdots \otimes e_{l_ni}.$$\end{proposition}
\begin{proof}It is a direct verification by using induction and Theorem 5.2 or Hoffman's formula.\end{proof}

\section{The dual construction}
In this section, we give a dual construction of the quantum
quasi-shuffle algebra by using the universal property of tensor
algebras. First of all, we study a special coalgebra structure on
$T(C)$ which is a generalization of the quantized cofree coalgebra
structure. For coalgebras, we adopt Sweedler's notation. That
means for a coalgebra $(C,\bigtriangleup,\varepsilon)$ and any
$c\in C$, we denote
$$\bigtriangleup(c)=\sum_{(c)}c_{(1)}\otimes c_{(2)},$$
or simply, $\bigtriangleup(c)=c_{(1)}\otimes c_{(2)}$.

To extend algebra structures to quantized case, one needs the
notion of braided algebras. By contrast, in the case of
coalgebras, one needs the so-called braided coalgebras.

\begin{definition}Let $C=(C,\bigtriangleup,\varepsilon)$ be a coalgebra with coproduct $\bigtriangleup$ and counit $\varepsilon$, and $\sigma$ be a braiding on $C$. We call $(C,\bigtriangleup,\sigma)$ a \emph{braided coalgebra} if it satisfies the following conditions:
\[
\begin{split}
(\mathrm{id}_C\otimes \bigtriangleup)\sigma&=\sigma_1\sigma_2(
\bigtriangleup\otimes
\mathrm{id}_C),\\[3pt] ( \bigtriangleup\otimes
\mathrm{id}_C)\sigma&=\sigma_2\sigma_1(\mathrm{id}_C\otimes
\bigtriangleup). \end{split}
\]

\end{definition}

In order to get new braided algebras and coalgebras from old ones,
we need the proposition below.
\begin{proposition}[\cite{HH}, Proposition 4.2]1.
For a braided algebra $(A,\mu,\sigma)$ and any $i\in \mathbb{N}$,
$(A^{\otimes i},\mu_{\sigma,i}, \beta_{ii})$ becomes a braided
algebra with product $\mu_{\sigma,i}=\mu^{\otimes i}\circ
T^\sigma_{w_i}$, where $w_i\in \mathfrak{S}_{2i}$ is given by
\[w_{i}=\left(\begin{array}{ccccccccc}
1&2&3&\cdots&i&i+1&i+2&\cdots & 2i\\
1&3&5&\cdots&2i-1&2& 4 &\cdots & 2i
\end{array}\right).\]

2. For a braided coalgebra $(C, \bigtriangleup,\sigma)$,
$(C^{\otimes i},\bigtriangleup_{\sigma,i}, \beta_{ii})$ becomes a
braided coalgebra with coproduct $\bigtriangleup_{\sigma,i}=
T^\sigma_{w_i^{-1}}\circ\bigtriangleup^{\otimes i}$ and counit
$\varepsilon^{\otimes i}:C^{\otimes i}\rightarrow
\mathbb{K}^{\otimes i}\simeq \mathbb{K}$.\end{proposition}

Let $(C,\bigtriangleup,\sigma)$ be a braided coalgebra. Consider
the tensor algebra $T(C)$ with the concatenation product $m$. Then
by
 the above proposition, $\mathscr{T} _\beta^2(C)=(T(C)\underline{\otimes}
 T(C),m_{\beta,2})$ and $\mathscr{T} _\beta^3=(T(C)\underline{\otimes}
 T(C)\underline{\otimes}
 T(C),m_{\beta,3})$ are associative algebras.

 We define
\[\begin{array}{cccc}
\phi_1: &C&\rightarrow& T(C)\underline{\otimes}
 T(C)
,\\[6pt]
&c&\mapsto&1\underline{\otimes} c+ c\underline{\otimes}1.
\end{array}\]

By the universal property of tensor algebras, there exists an
algebra map $\Phi_1: T(C)\rightarrow \mathscr{T} _\beta^2(C)$
whose restriction on $C$ is $\phi_1$. Moreover $\Phi_1$ is
coassociative and is the dual of quantum shuffle product (see,
e.g., \cite{FG}).

Now we define
\[\begin{array}{cccc}
\phi_2: &C&\rightarrow& T(C)\underline{\otimes}
 T(C)
,\\[6pt]
&c&\mapsto&\sum c_{(1)}\underline{\otimes} c_{(2)}.
\end{array}\]

By the universal property of tensor algebras again, there exists
an algebra map $\Phi_2: T(C)\rightarrow \mathscr{T} _\beta^2(C)$
whose restriction on $C$ is $\phi_2$.

\begin{proposition}For any $i\in \mathbb{N}$, we have that $\Phi_2\mid_{C^{\otimes
i}}=\bigtriangleup_{\sigma,i}$. So $(T(C), \Phi_2, \beta)$ is a
braided coalgebra.\end{proposition}
\begin{proof}We use induction on $i$. When $i=1$, it is trivial. Assume the equality holds for the case $i< n$.
Then for any $c_1,\ldots,c_n\in C$, we have
\begin{eqnarray*}\lefteqn{\Phi_2(c_1\otimes \cdots \otimes
c_n)}\\[3pt]
&=&m_{\beta,2}\big(\Phi_2(c_1)\otimes\Phi_2(c_2\otimes \cdots \otimes c_n)\big)\\[3pt]
&=&(\mathrm{id}_{T(C)}\otimes \beta \otimes \mathrm{id}_{T(C)})(\bigtriangleup(c_1)\otimes T^\sigma_{w_{n-1}^{-1}}\circ\bigtriangleup^{\otimes n-1}(c_2\otimes \cdots \otimes c_n)\big)\\[3pt]
&=&(\mathrm{id}_C\otimes T^\sigma_{\chi_{1,n-1}}\otimes
\mathrm{id}_C^{\otimes n-1})(\mathrm{id}_C^{\otimes 2}\otimes
T^\sigma_{w_{n-1}^{-1}})\bigtriangleup^{\otimes n}(c_1\otimes
\cdots \otimes
c_n)\\[3pt]
&=&T^\sigma_{w_n^{-1}}\circ\bigtriangleup^{\otimes n}(c_2\otimes
\cdots \otimes c_n),\end{eqnarray*}where the last equality follows
from the fact that $(1_{\mathfrak{G}_1 }\times\chi_{1,n-1}\times
1_{\mathfrak{G}_{}n-1 })(1_{\mathfrak{G}_2 }\times
w_{n-1}^{-1})=w_{n}^{-1}$ and the expression is
reduced.\end{proof}

Let $\phi=\phi_1+\phi_2: C\rightarrow\mathscr{T} _\beta^2(C)$ and
$\Phi$ be the algebraic map induced by the universal property of
tensor algebras which extends $\phi$.

\begin{proposition}Under the notation above, the triple $(T(C), \Phi, \beta)$ is a braided coalgebra.
\end{proposition}
\begin{proof}We first show that
\[\left\{
\begin{array}{lll}
\beta_1\beta_2 (\Phi\otimes \mathrm{id}_{T(C)})&=& (\mathrm{id}_{T(C)}\otimes \Phi)\beta,\\[3pt]
\beta_2\beta_1(\mathrm{id}_{T(C)}\otimes \Phi)&=&(\Phi\otimes
\mathrm{id}_{T(C)})\beta.
\end{array} \right.
\]
For any $x\in C^{\otimes i}$ and $y\in C^{\otimes j}$ we verify
the first one on $x\underline{\otimes}y$. The second one can be
verified similarly. We use induction on $i$.

When $i=1$,
\begin{eqnarray*}\beta_1\beta_2 (\Phi\otimes \mathrm{id}_{T(C)})(x\underline{\otimes}y)&=&\beta_1\beta_2 (\phi_1\otimes \mathrm{id}+\phi_2\otimes \mathrm{id})(x\underline{\otimes}y)\\[3pt]
&=&(\mathrm{id}\otimes \phi_1+\mathrm{id}\otimes \phi_2)\beta(x\underline{\otimes}y)\\[3pt]
&=& (\mathrm{id}_{T(C)}\otimes
\Phi)\beta(x\underline{\otimes}y).\end{eqnarray*} For any $c\in
C$, we have
\begin{eqnarray*}\beta_1\beta_2 (\Phi\otimes \mathrm{id}_{T(C)})\big((c\otimes x)\underline{\otimes}y\big)&=&\beta_1\beta_2\beta_3\beta_4\beta_2(\Phi\otimes \Phi\otimes\mathrm{id}_{T(C)})(c\underline{\otimes} x\underline{\otimes}y)\\[3pt]
&=&\beta_1\beta_3\beta_2\beta_3\beta_4(\Phi\otimes \Phi\otimes\mathrm{id}_{T(C)})(c\underline{\otimes} x\underline{\otimes}y)\\[3pt]
&=&\beta_1\beta_3\beta_2\big(\Phi\otimes \beta_1\beta_2(\Phi\otimes\mathrm{id}_{T(C)})\big)(c\underline{\otimes} x\underline{\otimes}y)\\[3pt]
&=&\beta_1\beta_3\beta_2(\Phi\otimes \mathrm{id}_{T(C)}\otimes\Phi)\beta_2(c\underline{\otimes} x\underline{\otimes}y)\\[3pt]
&=&\beta_3\big(\beta_1\beta_2(\Phi\otimes \mathrm{id}_{T(C)})\otimes\Phi\big)\beta_2(c\underline{\otimes} x\underline{\otimes}y)\\[3pt]
&=&\beta_3(\mathrm{id}_{T(C)}\otimes \Phi\otimes\Phi)\beta_1\beta_2(c\underline{\otimes} x\underline{\otimes}y)\\[3pt]
&=&(\mathrm{id}_{T(C)}\otimes \Phi)\beta\big((c\otimes
x)\underline{\otimes}y\big).\end{eqnarray*}

The next step is to show $(\Phi\otimes
\mathrm{id}_{T(C)})\Phi=(\mathrm{id}_{T(C)}\otimes \Phi)\Phi$.
Notice that for any $c\in C$,
\begin{eqnarray*}\lefteqn{(\Phi\otimes \mathrm{id}_{T(C)})\Phi(c)}\\[3pt]
&=&(\Phi\otimes \mathrm{id}_{T(C)})(c_{(1)}\underline{\otimes} c_{(2)}+1\underline{\otimes} c+c\underline{\otimes} 1)\\[3pt]
&=&c_{(1)}\underline{\otimes} c_{(2)}\underline{\otimes}
c_{(3)}+1\underline{\otimes} c_{(1)}\underline{\otimes} c_{(2)}+
c_{(1)}\underline{\otimes} c_{(2)}\underline{\otimes}
1\\[3pt]
&&+c_{(1)}\underline{\otimes}1 \underline{\otimes}
c_{(2)}+1\underline{\otimes} 1
\underline{\otimes}c+c_{(1)}\underline{\otimes} c_{(2)}\underline{\otimes} 1+1\underline{\otimes}c \underline{\otimes}1+c\underline{\otimes}1\underline{\otimes}1\\[3pt]
&=&(\mathrm{id}_{T(C)}\otimes \Phi)\Phi(c).\end{eqnarray*} By the
uniqueness of the universal property of $T(C)$, we only need to
show that both $(\Phi\otimes \mathrm{id}_{T(C)})\Phi$ and
$(\mathrm{id}_{T(C)}\otimes \Phi)\Phi$ are algebra morphisms from
$T(C)$ to $\mathscr{T} _\beta^3$. Since $\Phi: T(C)\rightarrow
\mathscr{T} _\beta^2(C)$ is an algebra morphism, we have that
$$\Phi\circ m=(m\otimes m)(\mathrm{id}_{T(C)}\otimes \beta
\otimes \mathrm{id}_{T(C)})(\Phi\otimes \Phi).$$ So
\begin{eqnarray*}
\lefteqn{(\Phi\otimes \mathrm{id}_{T(C)})\circ\Phi\circ
m}\\[3pt]
&=&(\Phi\otimes \mathrm{id}_{T(C)})(m\otimes
m)(\mathrm{id}_{T(C)}\otimes \beta \otimes
\mathrm{id}_{T(C)})(\Phi\otimes \Phi)\\[3pt]
&=&\big((\Phi\circ m)\otimes m\big)(\mathrm{id}_{T(C)}\otimes
\beta \otimes
\mathrm{id}_{T(C)})(\Phi\otimes \Phi)\\[3pt]
&=&(m\otimes m \otimes m)(\mathrm{id}_{T(C)}\otimes \beta \otimes
\mathrm{id}_{T(C)} \otimes \mathrm{id}_{T(C)}\otimes
\mathrm{id}_{T(C)})\\[3pt]
&&\circ(\Phi\otimes \Phi  \otimes \mathrm{id}_{T(C)} \otimes
\mathrm{id}_{T(C)})(\mathrm{id}_{T(C)}\otimes \beta \otimes
\mathrm{id}_{T(C)})(\Phi\otimes \Phi)\\[3pt]
&=&(m\otimes m\otimes m) \beta_2\\[3pt]
&&\circ\big(\mathrm{id}_{T(C)} \otimes \mathrm{id}_{T(C)}\otimes(\Phi\otimes \mathrm{id}_{T(C)}) \beta \otimes \mathrm{id}_{T(C)}\big)\big((\Phi\otimes \mathrm{id}_{T(C)})\Phi\otimes \Phi\big)\\[3pt]
&=&(m\otimes m\otimes
m)\beta_2\beta_4\beta_3\\[3pt]
&&\circ(\mathrm{id}_{T(C)}\otimes \mathrm{id}_{T(C)}\otimes
\mathrm{id}_{T(V)}\otimes \Phi\otimes
\mathrm{id}_{T(C)})\big((\Phi\otimes \mathrm{id}_{T(C)})\Phi\otimes \Phi\big)\\[3pt]
&=&m_{\beta,3}\big((\Phi\otimes \mathrm{id}_{T(C)})\Phi\otimes(
\Phi_2\otimes \mathrm{id}_{T(C)})\Phi\big).
\end{eqnarray*}
It follows that $(\Phi\otimes \mathrm{id}_{T(C)})\Phi$ is an
algebra morphism. Similarly, the map $(\mathrm{id}_{T(C)}\otimes
\Phi)\Phi$ is also an algebra morphism.
\end{proof}

Now we begin to study the relation between the braided coalgebra
$(T(C), \Phi, \beta)$ and the quantum quasi-shuffle algebra. We
show that they are dual to each other in the following sense.

Let $<,>: V\times W\rightarrow \mathbb{K}$ and
$<,>^\prime:V^\prime\times W^\prime\rightarrow \mathbb{K}$ be two
bilinear non-degenerate forms. For any $f\in \mathrm{Hom}(V,
V^\prime)$, the adjoint operator $\mathrm{adj}(f)\in
\mathrm{Hom}(W^\prime, W)$ of $f$ is defined to be the one such
that $<x,\mathrm{adj}(f)(y)>=<f(x),y>^\prime$ for any $x\in V$ and
$y\in W^\prime$. It is clear that $\mathrm{adj}(f\circ
g)=\mathrm{adj}(g)\circ\mathrm{adj}(f)$.

\begin{remark}If there is a non-degenerate bilinear form between two vector spaces $A$ and $C$, and $(A,m,\sigma)$ is a braided algebra, then
$(C,\mathrm{adj}(m),\mathrm{adj}(\sigma))$ is a braided coalgebra.
The converse is also true.\end{remark}

From now on, we assume that there always exists a non-degenerate
bilinear form $<,>$ between two vector spaces $A$ and $C$. It can
be extended to a bilinear form $<,>: A^{\otimes n}\times
C^{\otimes n}\rightarrow \mathbb{K}$ for any $n \geq 1$ in the
usual way: for any $a_1,\ldots, a_n\in A$ and $c_1,\ldots,c_n\in
C$,
$$<a_1\otimes \cdots \otimes a_n, c_1\otimes \cdots \otimes
c_n>=\prod_{i=1}^n <a_i, c_i>.$$ It induces a non-degenerate
bilinear form  $<,>: T(A)\times T(C)\rightarrow \mathbb{K}$ by
setting that $<x,y>=0$ for any $x\in A^{\otimes i}$ , $y\in
C^{\otimes j}$ and $i\neq j$. Then we can define a non-degenerate
bilinear form $<,>: T(A)\underline{\otimes}T(A)\times
T(C)\underline{\otimes}T(C)\rightarrow \mathbb{K}$ by requiring
that $<u\underline{\otimes} v,x\underline{\otimes}y>=<u,x><v,y>$
for any $u,v\in T(A)$ and $x,y\in T(C)$.

If $(C,\bigtriangleup,\sigma)$ is a braided coalgebra, we denote
$\tau=\mathrm{adj}(\sigma)$ and $\alpha=\mathrm{adj}(\beta)$. Then
$\alpha$ is a braiding on $T(A)$ and
$\alpha_{i,j}=T^\tau_{\chi_{ji}^{-1}}=T^\tau_{\chi_{ij}}$.

\begin{theorem}Under the assumptions above, we have that $\mathrm{adj}(\Phi)=\Join_\sigma$.\end{theorem}
\begin{proof}For any $c_1,\ldots,c_n\in
C$, we notice that
\begin{eqnarray*}\Phi(c_1\otimes \cdots\otimes c_n)&=&m_{\beta,2}\big(\Phi(c_1)\underline{\otimes}\Phi(c_2\otimes \cdots\otimes c_n\big)\\[3pt]
&=&m_{\beta,2}\big((1\underline{\otimes}c_1)\underline{\otimes}\Phi(c_2\otimes \cdots\otimes c_n)\big)\\[3pt]
&&+m_{\beta,2}\big((v_1\underline{\otimes}1)\underline{\otimes}\Phi(c_2\otimes \cdots\otimes c_n)\big)\\[3pt]
&&+m_{\beta,2}\big((c_{1(1)}\underline{\otimes}c_{1
(2)})\underline{\otimes}\Phi(c_2\otimes \cdots\otimes
c_n)\big)\\[3pt]
&=&(\beta_{1,?}\otimes \mathrm{id}_C^{\otimes n-1-?})(\mathrm{id}_C\otimes \Phi)(c_1\otimes \cdots\otimes c_n)\\[3pt]
&&+(\mathrm{id}_C\otimes \Phi)(c_1\otimes \cdots\otimes c_n)\\[3pt]
&&+(\mathrm{id}_C\otimes\beta_{1,?}\otimes \mathrm{id}_C^{\otimes
n-1-?})(\bigtriangleup\otimes \Phi)(c_1\otimes \cdots\otimes
c_n).\end{eqnarray*} Here, since $\Phi(C^{\otimes k})\subset
C^{\otimes 0}\underline{\otimes}C^{\otimes k}+C^{\otimes
k}\underline{\otimes}C^{\otimes 0}+\cdots+C^{\otimes
k}\underline{\otimes}C^{\otimes k}$, we denote by $\beta_{1,?}$
the action of $\beta$ on $C\underline{\otimes}C^{\otimes ?}$ where
$C^{\otimes ?}$ is the left factor of some component in
$\Phi(C^{\otimes k})$.

We denote by $\mathrm{adj}(\Phi)_{(k,l)}$ the action of
$\mathrm{adj}(\Phi)$ on $A^{\otimes
k}\underline{\otimes}A^{\otimes l}$. Then on $A^{\otimes
i}\underline{\otimes}A^{\otimes j}$, we have
\begin{eqnarray*}\mathrm{adj}(\Phi)_{(i,j)}&=&\mathrm{adj}\big((\beta_{1,?}\otimes
\mathrm{id}_C^{\otimes i+j-1-?})(\mathrm{id}_C\otimes
\Phi)+(\mathrm{id}_C\otimes
\Phi)\\[3pt]
&&\ \ \ \ \ \ +(\mathrm{id}_C\otimes\beta_{1,?}\otimes
\mathrm{id}_C^{\otimes i+j-1-?})(\bigtriangleup\otimes \Phi)\big)_{(i,j)}\\[3pt]
&=&\big(\mathrm{id}_A\otimes \mathrm{adj}(\Phi)_{(i,j-1)}\big)(\alpha_{i,1}\otimes \mathrm{id}_A^{\otimes j-1})\\[3pt]
&&+\big(\mathrm{id}_A\otimes \mathrm{adj}(\Phi)_{(i-1,j)}\big)\\[3pt]
&&+\big(\mathrm{adj}(\bigtriangleup)\otimes
\mathrm{adj}(\Phi)_{(i-1,j-1)}\big)(\mathrm{id}_A\otimes\beta_{i-1,1}\otimes
\mathrm{id}_A^{\otimes j-1}).\end{eqnarray*}This shows that the
map $\mathrm{adj}(\Phi)$ shares the same inductive formula with
the quantum quasi-shuffle product built on $(A,
\mathrm{adj}(\bigtriangleup),\tau)$. Hence we have the conclusion.
\end{proof}

\section{Representations from Rota-Baxter algebras}

In order to provide representations of quantum quasi-shuffle
algebras, our main point is to construct algebra homomorphisms
from quantum quasi-shuffles to some algebras. Inspired by the
works on the connection between mixable shuffle algebras and
Rota-Baxter algebras (\cite{GK} and \cite{G}), we choose the
target algebras to be relevant object of Rota-Baxter algebras in
braided categories with some suitable conditions. We first recall
the definition of Rota-Baxter algebras.
\begin{definition}Let $\lambda$ be an element in $\mathbb{K}$. A pair $(R,P)$ is
called a \emph{Rota-Baxter algebra of weight $\lambda$} if $R$ is
an algebra and $P$ is a linear endomorphism of $R$ satisfying that
for any $x,y\in R$,
$$P(x)P(y)=P(xP(y))+P(P(x)y)+\lambda P(xy).$$ Such a map $P$ is called a \emph{Rota-Baxter operator of weight $\lambda$}.\end{definition}

We extend this notion in braided categories.

\begin{definition}A triple $(R,P,\sigma)$ is called a \emph{right (resp. left) weak braided Rota-Baxter algebra of weight $\lambda$} if $(R,\sigma)$ is a braided algebra and $P$ is a Rota-Baxter operator on $R$ of weight $\lambda$ such that $\sigma(P\otimes \mathrm{id}_R)=(\mathrm{id}_R\otimes P)\sigma$ (resp. $\sigma(\mathrm{id}_R\otimes P)=(P\otimes \mathrm{id}_R)\sigma$). If the multiplication $m$ of $R$ satisfies $m\circ \sigma=m$, $R$ is said to be \emph{commutative}.\end{definition}

Obviously, any usual Rota-Baxter algebra is right weak braided
with respect to the usual flip map. Note that a right weak braided
Rota-Baxter algebra which is also left weak is a braided
Rota-Baxter algebra (cf. \cite{J}). Examples of right or left weak
braided Rota-Baxter algebras can be found in \cite{J}.

\begin{theorem}Suppose that $(A, m,\sigma)$ is a braided algebra, $(R,P,\sigma')$ is a commutative right weak braided Rota-Baxter algebra of weight 1, $f: A\rightarrow R$ is an algebra map such that $(f\otimes f)\circ \sigma=\sigma'\circ(f\otimes f)$. We define
a linear map $\overline{f}:T(A)\rightarrow R$ recursively by
$\overline{f}(1)=1_R$, $\overline{f}(a_1)=P(f(a_1))$ and
$\overline{f}(a_1\otimes \cdots\otimes
a_n)=P\big(f(a_1)\overline{f}(a_2\otimes \cdots\otimes a_n)\big)$
for $n\geq 2$. Then $\overline{f}$ is an algebra map from the
quantum quasi-shuffle algebra $T_{\sigma}^{qsh}(A)$ to
$R$.\end{theorem}
\begin{proof}We denote by $m'$ the multiplication of $R$. For any $i,j\geq 0$, we will verify that $\overline{f}\Join_\sigma =m' (\overline{f}\otimes \overline{f})$ on $A^{\otimes i}\underline{\otimes}A^{\otimes j}$. We use induction on
$n=i+j$.

When $i=0$ or $j=0$, it is trivial. So we assume that $i\neq 0$
and $j\neq 0$ in the sequel.

For $n=2$, i.e., $i=j=1$, assume $a,b\in A$, then\begin{eqnarray*}
\lefteqn{\overline{f}(a\Join_\sigma b)}\\[3pt]
&=&\overline{f}\big(a\otimes b+\sigma(a\otimes b)+m(a\otimes b)\big)\\[3pt]
&=&P\big(f(a)\overline{f}(b)\big)+P m' (f\otimes Pf) \sigma(a\otimes b)+P\big(f(ab)\big)\\[3pt]
&=&P\Big(f(a)P\big(f(b)\big)\Big)+P m' (\mathrm{id}_R\otimes P) \sigma'(f\otimes f)(a\otimes b)+P\big(f(a)f(b)\big)\\[3pt]
&=&P\Big(f(a)P\big(f(b)\big)\Big)+Pm' \sigma'(P\otimes \mathrm{id}_R)\big(f(a)\otimes f( b)\big)+P\big(f(a)f(b)\big)\\[3pt]
&=&P\Big(f(a)P\big(f(b)\big)\Big)+P m'(P\otimes \mathrm{id}_R)\big(f(a)\otimes f( b)\big)+P\big(f(a)f(b)\big)\\[3pt]
&=&P\Big(f(a)P\big(f(b)\big)\Big)+P\Big(P\big( f(a)\big)f(b)\Big)+P\big(f(a)f(b)\big)\\[3pt]
&=&P\big(f(a)\big)P\big(f(b)\big)\\[3pt]
&=&\overline{f}(a)\overline{f}(b).
\end{eqnarray*}

Assume that $n>2$ and the result holds for $i+j<n$. Note that by
an easy induction we have $(f\otimes
\overline{f})\beta_{k,1}=\sigma'(\overline{f}\otimes f)$ for any
$k\in \mathbb{N}$.

For $i=n-2$ and $x\in A^{\otimes i}$,\begin{eqnarray*}
\lefteqn{\overline{f}\big((a\otimes x)\Join_\sigma b\big)}\\[3pt]
&=&\overline{f}\big(a\otimes (x\Join_\sigma
b)+\beta_{i+1,1}(a_1\otimes x\otimes
b)\\[3pt]
&&+(m\otimes\mathrm{id}_A^{\otimes  i})(\mathrm{id}_A\otimes
\beta_{i,1})(a\otimes x\otimes b)\big)\\[3pt]
&=&P\big(f(a)\overline{f}(x\Join_\sigma b)\big)\\[3pt]
&&+P m' (\mathrm{id}_R\otimes Pm')(f\otimes f\otimes
\overline{f})(\sigma\otimes \mathrm{id}_A^{\otimes
i})(\mathrm{id}_A\otimes \beta_{i,1})(a_1\otimes x\otimes
b) \\[3pt]
&&+Pm'(fm\otimes\overline{f})(\mathrm{id}_A\otimes
\beta_{i,1})(a\otimes x\otimes b)\\[3pt]
&=&P\big(f(a)\overline{f}(x)\overline{f}( b)\big)\\[3pt]
&&+P m' (\mathrm{id}_R\otimes Pm')(\sigma'\otimes
\mathrm{id}_R)\big(f\otimes (f\otimes
\overline{f})\beta_{i,1}\big)(a_1\otimes x\otimes
b) \\[3pt]
&&+Pm'\big(m'(f\otimes
f)\otimes\overline{f}\big)(\mathrm{id}_A\otimes
\beta_{i,1})(a\otimes x\otimes b)\\[3pt]
&=&P\Big(f(a)\overline{f}(x)P\big(f(b)\big)\Big)\\[3pt]
&&+P m' (\mathrm{id}_R\otimes Pm')(\sigma'\otimes
\mathrm{id}_R)\big(f\otimes \sigma'(\overline{f}\otimes
f)\big)(a_1\otimes x\otimes
b) \\[3pt]
&&+Pm'(m'\otimes\mathrm{id}_R)\big(f\otimes
(f\otimes\overline{f})\beta_{i,1}\big)(a\otimes x\otimes b)\\[3pt]
&=&P\Big(f(a)\overline{f}(x)P\big(f(b)\big)\Big)+P m'
(\mathrm{id}_R\otimes P)(\mathrm{id}_R\otimes
m')\sigma'_1\sigma'_2 \big(f(a_1)\otimes \overline{f}(x)\otimes f(b)\big)\\[3pt]
&&+Pm'(\mathrm{id}_R\otimes m')\big(f\otimes
\sigma'(\overline{f}\otimes f)\big)(a\otimes x\otimes b)\\[3pt]
&=&P\Big(f(a)\overline{f}(x)P\big(f(b)\big)\Big)+P m'
(\mathrm{id}_R\otimes P)\sigma'(m'\otimes\mathrm{id}_R
)\big(f(a_1)\otimes \overline{f}(x)\otimes f(b)\big)\\[3pt]
&&+Pm'(\mathrm{id}_R\otimes m'\sigma')(f\otimes
\overline{f}\otimes f)(a\otimes x\otimes b)\\[3pt]
&=&P\Big(f(a)\overline{f}(x)P\big(f(b)\big)\Big)+P(P\big(f(a)\overline{f}(x))f(b)\big)+P\big(f(a)\overline{f}(y)f(b)\big)\\[3pt]
&=&P\big(f(a)\overline{f}(x)\big)P\big(f(b)\big)\\[3pt]
&=&\overline{f}(a\otimes x)\overline{f}(b).
\end{eqnarray*}

For $j=n-2$ and $y\in A^{\otimes j}$, \begin{eqnarray*}
\lefteqn{\overline{f}\big(a\Join_\sigma (b\otimes y)\big)}\\[3pt]
&=&\overline{f}\big(a\otimes b\otimes y+(\mathrm{id}_A\otimes\Join_{\sigma  (1,j)})(\beta_{1,1}\otimes\mathrm{id}_A^{\otimes  j})(a\otimes b\otimes y)+(ab)\otimes y\big)\\[3pt]
&=&P\Big(f(a)P\big(f(b)\overline{f}(y)\big)\Big)+P m' (f\otimes \overline{f} \Join_{\sigma  (1,j)})(\sigma\otimes\mathrm{id}_A^{\otimes  j})(a\otimes b\otimes y) \\[3pt]
&&+P\big(f(ab)\overline{f}(y)\big)\\[3pt]
&=&P\Big(f(a)P\big(f(b)\overline{f}(y)\big)\Big)+P m' \big(f\otimes m'(\overline{f}\otimes \overline{f}) \big)(\sigma\otimes\mathrm{id}_A^{\otimes  j})(a\otimes b\otimes y) \\[3pt]
&&P\big(f(a)f(b)\overline{f}(y)\big)\\[3pt]
&=&P\Big(f(a)P\big(f(b)\overline{f}(y)\big)\Big)+P\big(f(a)f(b)\overline{f}(y)\big)\\[3pt]
&&+P m'(\mathrm{id}_R\otimes m') (f\otimes Pf\otimes \overline{f}) (\sigma\otimes\mathrm{id}_A^{\otimes  j})(a\otimes b\otimes y)\\[3pt]
&=&P\Big(f(a)P\big(f(b)\overline{f}(y)\big)\Big)+P\big(f(a)f(b)\overline{f}(y)\big)\\[3pt]
&&+P m'(m'\otimes \mathrm{id}_R) \big((f\otimes Pf)\sigma\otimes \overline{f}\big) (a\otimes b\otimes y)\\[3pt]
&=&P\Big(f(a)P\big(f(b)\overline{f}(y)\big)\Big)+P\big(f(a)f(b)\overline{f}(y)\big)\\[3pt]
&&+P m'(m'\sigma' \otimes \mathrm{id}_R)(P\otimes\mathrm{id}_R^{\otimes 2})(f\otimes f\otimes \overline{f})(a\otimes b\otimes y)\\[3pt]
&=&P\Big(f(a)P\big(f(b)\overline{f}(y)\big)\Big)+P\Big(P\big(f(a)\big)f(b)\overline{f}(y)\Big)+P\big(f(a)f(b)\overline{f}(y)\big)\\[3pt]
&=&P\big(f(a)\big)P\big(f(b)\overline{f}(y)\big)\\[3pt]
&=&\overline{f}(a)\overline{f}(b\otimes y).
\end{eqnarray*}

Finally, for $i,j\geq 1$ with $i+j=n-2$ and $x\in A^{\otimes i},
y\in A^{\otimes j}$,\begin{eqnarray*}
\lefteqn{\overline{f}\big((a\otimes x)\Join_\sigma (b\otimes y)\big)}\\[3pt]
&=&\overline{f}\big(a\otimes (x\Join_{\sigma}(b\otimes
y))+(\mathrm{id}_A\otimes  \Join_{\sigma
(i+1,j)})(\beta_{i+1,1}\otimes \mathrm{id}_A^{\otimes j})(a\otimes
x\otimes  b\otimes y)\\[3pt]
&&+(m\otimes \Join_{\sigma (i,j)} )(\mathrm{id}_A\otimes
\beta_{i,1}\otimes \mathrm{id}_A^{\otimes j})(a\otimes
x\otimes b\otimes y)\\[3pt]
&=&P\Big(f(a)\overline{f}\big (x\Join_{\sigma}(b\otimes
y)\big)\Big)\\[3pt]
&&+Pm'(f\otimes \overline{f}\Join_{\sigma
(i+1,j)})(\beta_{i+1,1}\otimes \mathrm{id}_A^{\otimes j})(a\otimes
x\otimes  b\otimes y)\\[3pt]
&&+Pm'(fm\otimes \overline{f}\Join_{\sigma (i,j)}
)(\mathrm{id}_A\otimes \beta_{i,1}\otimes \mathrm{id}_A^{\otimes
j})(a\otimes
x\otimes b\otimes y)\\[3pt]
&=&P\big(f(a)\overline{f} (x)\overline{f}(b\otimes
y)\big)\\[3pt]
&&+Pm'\big(f\otimes m'(\overline{f}\otimes
\overline{f})\big)(\beta_{i+1,1}\otimes \mathrm{id}_A^{\otimes
j})(a\otimes
x\otimes  b\otimes y)\\[3pt]
&&+Pm'\big(m'(f\otimes f)\otimes m'(\overline{f}\otimes
\overline{f}) \big)(\mathrm{id}_A\otimes \beta_{i,1}\otimes
\mathrm{id}_A^{\otimes j})(a\otimes
x\otimes b\otimes y)\\[3pt]
&=&P\Big(f(a)\overline{f} (x)P\big(f(b)\overline{f}(
y)\big)\Big)\\[3pt]
&&+Pm'(\mathrm{id}_R\otimes m')\big((f\otimes
\overline{f})\beta_{i+1,1}\otimes \overline{f}\big)(a\otimes
x\otimes  b\otimes y)\\[3pt]
&&+Pm'(m'\otimes m')\big(f\otimes (f\otimes
\overline{f})\beta_{i,1}\otimes \overline{f}\big) (a\otimes
x\otimes b\otimes y)\\[3pt]
&=&P\Big(f(a)\overline{f} (x)P\big(f(b)\overline{f}(
y)\big)\Big)\\[3pt]
&&+Pm'(m'\otimes \mathrm{id}_R)\big(\sigma'(\overline{f}\otimes f
)\otimes \overline{f}\big)(a\otimes
x\otimes  b\otimes y)\\[3pt]
&&+Pm'(m'\otimes m')\big(f\otimes \sigma'(\overline{f}\otimes f
)\otimes \overline{f}\big) (a\otimes
x\otimes b\otimes y)\\[3pt]
&=&P\Big(f(a)\overline{f} (x)P\big(f(b)\overline{f}(
y)\big)\Big)+P\big(\overline{f}(a\otimes
x)f(b) \overline{f}(y)\big)\\[3pt]
&&+P\big(f(a)
\overline{f}(x) f(b) \overline{f}(y)\big)\\[3pt]
&=&P\Big(f(a)\overline{f} (x)P\big(f(b)\overline{f}(
y)\big)\Big)+P\Big(P\big(f(a)
\overline{f}(x)\big)f(b) \overline{f}(y)\Big)\\[3pt]
&&+P\big(f(a)
\overline{f}(x) f(b) \overline{f}(y)\big)\\[3pt]
&=&P\big(f(a)\overline{f}(x)\big)P\big(f(b)\overline{f}(y)\big)\\[3pt]
&=&\overline{f}(a\otimes x)\overline{f}(b\otimes y),
\end{eqnarray*}
where the sixth equality follows from the fact that $m'(m'\otimes
m')(\mathrm{id}_R\otimes \sigma'\otimes\mathrm{id}_R
)=m'(m'\otimes m')$.
\end{proof}

\begin{remark}1. The quantum quasi-shuffle algebra $T_{\sigma}^{qsh}(A)$ built on a braided algebra $(A,m,\sigma)$ has another inductive formula (cf. \cite{JRZ}): for $i,j>1$ and any $a_1,\ldots,
a_i,b_1,\ldots, b_j\in A$, we have
\begin{eqnarray*}
\lefteqn{(a_1\otimes\cdots\otimes a_i)\Join_\sigma (b_1\otimes\cdots\otimes
b_j)}\\[3pt]
&=&\Big((a_1\otimes\cdots\otimes a_i)\Join_\sigma (b_1\otimes\cdots\otimes b_{j-1})\Big)\otimes b_j\\[3pt]
&&+(\Join_{\sigma (i-1,j)}\otimes
\mathrm{id}_A)(\mathrm{id}_A^{\otimes
i-1}\otimes\beta_{1j})(a_1\otimes\cdots\otimes a_i\otimes b_1\otimes\cdots\otimes b_j)\\[3pt]
&&+(\Join_{\sigma (i-1,j-1)}\otimes m)(\mathrm{id}_A^{\otimes
i-1}\otimes\beta_{1,j-1}\otimes\mathrm{id}_A)(a_1\otimes\cdots\otimes
a_i\otimes b_1\otimes\cdots\otimes b_j).
\end{eqnarray*}If $(R,P,m',\sigma')$ is a commutative left weak braided Rota-Baxter algebra of weight 1, $f: A\rightarrow R$ is an algebra map such that $(f\otimes f) \sigma=\sigma'(f\otimes f)$, then the linear map $\overline{f}:T(V)\rightarrow R$ defined recursively by $\widetilde{f}=Pm'(\widetilde{f}\otimes f)$ is also an algebra map from the
quantum quasi-shuffle algebra $T_{\sigma}^{qsh}(A)$ to $R$. The
proof is similar to the one of the above theorem.

2. By composing the left regular representation of $R$, Theorem
7.3 provides a representation of $T_{\sigma}^{qsh}(A)$ on $R$.
Since $\mathrm{Im}\overline{f}\subset P(R)$ and $P(R)$ is a
subalgebra of $R$, $P(R)$ is a $T_{\sigma}^{qsh}(A)$-submodule of
$R$. Similarly, $\mathrm{Im}\overline{f}$ is
$T_{\sigma}^{qsh}(A)$-submodule of $P(R)$.

3. Restricting the map $\overline{f}$ to the subalgebra
$S^{qsh}_\sigma(A)$ of $T_{\sigma}^{qsh}(A)$, we obtain an algebra
map from $S^{qsh}_\sigma(A)$ to $R$, and hence a representation of
$S^{qsh}_\sigma(A)$ on $R$.\end{remark}

\begin{example}Let $(A,m,\sigma)$ be a braided algebra such that $m\sigma=m$. We define $P:A\rightarrow A$ by $P(a)=-a$. One can verify that $(A,m,P,\sigma)$ is a commutative right weak braided Rota-Baxter algebra of weight 1. If we define $f$ to be the identity map on $A$, then the map $\overline{f}:T_{\sigma}^{qsh}(A)\rightarrow A$ given by $\overline{f}(a_1\otimes\cdots \otimes a_n)=(-1)^na_1\cdots a_n$ is an algebra map. So $T_{\sigma}^{qsh}(A)$ acts on $A$ by $(a_1\otimes\cdots \otimes a_n)\cdot a=(-1)^na_1\cdots a_na$. In addition, if  $A$ is unital, then it follows immediately that $A$ is a cyclic $T_{\sigma}^{qsh}(A)$-module generated by $1_A$.\end{example}

\begin{example}Let $q\in \mathbb{K}$ be an invertible number which is not a root
of unity. Consider the algebra $A=\mathbb{K}[t]_+$ of polynomials without constant terms as a
commutative braided algebra with the flip $\sigma(t^n\otimes
t^m)=t^m\otimes t^n$. Then the quantum quasi-shuffle algebra built
on $A$ is just the usual quasi-shuffle algebra
$T^{qsh}(A)$ of $A$. We define
$P:A\rightarrow A$ by
$P(t^n)=\frac{q^nt^n}{1-q^n}$. Then $P$ is a Rota-Baxter operator
of weight 1 on $A$ (cf. \cite{G2}). Moreover,
$(A,P,\sigma)$ is a commutative right weak braided
Rota-Baxter algebra of weight 1. If we define $f$ to be the
identity map on $A$, then the map $\overline{f}$ from
$T^{qsh}(A)_+=T^{qsh}(A)-\mathbb{K}$ to $A$
given by $$\overline{f}(t^{i_1}\otimes\cdots \otimes
t^{i_n})=\frac{q^{ni_n+(n-1)i_{n-1}+\cdots +i_1}t^{i_1+\cdots
i_n}}{(1-q^{i_n})(1-q^{i_n+i_{n-1}})\cdots
(1-q^{i_n+\cdots+i_1})}$$ is an algebra map. So
$T^{qsh}(A)_+$ acts on $A$ by
$$(t^{i_1}\otimes\cdots \otimes t^{i_n})\cdot
t^m=\frac{q^{ni_n+(n-1)i_{n-1}+\cdots +i_1}t^{i_1+\cdots
i_n+m}}{(1-q^{i_n})(1-q^{i_n+i_{n-1}})\cdots
(1-q^{i_n+\cdots+i_1})}.$$

As a consequence,
for any $k$, $\mathbb{K}[t]t^k$ is again a
$T^{qsh}(A)_+$-module.

More generally, the ring $\mathbb{K}[x_1,\ldots,x_n]_+$ of polynomials in $n$ variables without constant terms
admits a Rota-Baxter operator $P$ of weight 1 given by
$P(x_1^{l_1}\cdots
x_n^{l_n})=\frac{q^{l_1+\cdots+l_n}}{1-q^{l_1+\cdots+l_n}}x_1^{l_1}\cdots
x_n^{l_n}$. Therefore, $T^{qsh}(\mathbb{K}[x_1,\ldots,x_n]_+)_+$ acts
on $\mathbb{K}[x_1,\ldots,x_n]_+$.
\end{example}

Finally, by using Example 7.6, we reformulate the expression of
multiple $q$-zeta values. We will replace every term in the
multiple $q$-zeta values by an image of the algebra map
$\overline{f}$ on some element of the subalgebra of quasi-shuffle
algebra $T^{qsh}(A)_+$ generated by $A$.

From now on, we assume $q\in \mathbb{K}$ is an invertible number
which is not a root of unity.

Let $(A, P)$ be the commutative Rota-Baxter algebra
given in Example 7.6. Consider the subalgebra
$S^{qsh}(A)$ of $T^{qsh}(A)_+$ generated by
$A$. We denote by $\ast$ the usual quasi-shuffle
product on $T^{qsh}(A)_+$. Note that for any $n,i\in
\mathbb{N}$, we have
$$\overline{f}\big((t^{n})^{\ast
i}\big)=\overline{f}(t^{n})^i=\frac{q^{ni}t^{ni}}{(1-q^n)^i}.$$
Given $k\in\mathbb{N}$ and any multi-index $(i_1,\ldots,i_k)\in
\mathbb{N}^k$, we set
$$Z_q(i_1,\ldots,i_k;t)=\sum_{n_1>\cdots>n_k>0}\overline{f}\big((t^{n_1})^{\ast i_1}\ast \cdots \ast (t^{n_k})^{\ast i_k}\big),$$where the sum is
over all multi-index $(n_1,\ldots, n_k)\in \mathbb{N}^k$ with the
indicated inequalities. Then after evaluating at 1, we have
$$Z_q(i_1,\ldots,i_k)=Z_q(i_1,\ldots,i_k;1)=\sum_{n_1>\cdots>n_k>0}\frac{q^{i_1n_1}\cdots q^{i_kn_k}}{(1-q^{n_1})^{i_1}\cdots
(1-q^{n_k})^{i_k}}.$$ The series $Z_q(i_1,\ldots,i_k)$ is
Zudilin's $q$-analogue of multiple zeta values (cf. \cite{Zu}).

Now we turn to Bradley's multiple $q$-zeta values (\cite{B}).
Given $k\in\mathbb{N}$, a multi-index $(i_1,\ldots,i_k)\in
\mathbb{N}^k$ is called \emph{admissible} if $i_1\geq 2$. For any
admissible index $(i_1,\ldots,i_k)$, the \emph{multiple $q$-zeta
value} $\zeta(i_1,\ldots,i_k)$ is defined as
follows:$$\zeta_q(i_1,\ldots,i_k)=\sum_{n_1>\cdots>n_k\geq
1}\frac{q^{(i_1-1)n_1}\cdots q^{(i_k-1)n_k}}{[n_1]_q^{i_1}\cdots
[n_k]_q^{i_k}}.$$

By using the algebra map $\overline{f}$, we define a formal series
$\zeta_q(i_1,\ldots,i_k;t)\in \mathbb{K}[[t]]$ as follows:
$$\zeta_q(i_1,\ldots,i_k;t)=\sum_{n_1>\cdots>n_k\geq
1}\frac{\overline{f}\Big(\big((1-q)t^{n_1}\big)^{\ast i_1}\ast
\cdots \ast \big((1-q)t^{n_k}\big)^{\ast
i_k}\Big)}{q^{n_1+\cdots+n_k}},$$where the sum is over all
multi-index $(n_1,\ldots, n_k)\in \mathbb{N}^k$ with the indicated
inequalities. By evaluating at 1, we have that
$\zeta_q(i_1,\ldots,i_k;1)=\zeta_q(i_1,\ldots,i_k)$.

\section*{Acknowledgements}
 I am grateful to Marc Rosso from whom I learned the original construction of quantum quasi-shuffles. This work was partially supported by National Natural
Science Foundation of China (Grant No. 11201067).

\bibliographystyle{amsplain}

\end{document}